\documentclass[fleqn,reqno,11pt,a4paper,final]{amsart}

\usepackage[a4paper,left=30mm,right=30mm,top=30mm,bottom=30mm,marginpar=20mm]{geometry}
\usepackage{amsmath}
\usepackage{amssymb}
\usepackage{amsthm}
\usepackage{amscd}
\usepackage[ansinew]{inputenc}
\usepackage{cite}
\usepackage{bbm}
\usepackage{color}
\usepackage[english=american]{csquotes}
\usepackage[final]{graphicx}
\usepackage{hyperref}
\usepackage{calc}
\usepackage{mathptmx}
\usepackage{t1enc}

\graphicspath{{../Pictures/}}

\numberwithin{equation}{section}

\newtheoremstyle{thmlemcorr}{10pt}{10pt}{\itshape}{}{\bfseries}{.}{10pt}{{\thmname{#1}\thmnumber{ #2}\thmnote{ (#3)}}}
\newtheoremstyle{thmlemcorr*}{10pt}{10pt}{\itshape}{}{\bfseries}{.}\newline{{\thmname{#1}\thmnumber{ #2}\thmnote{ (#3)}}}
\newtheoremstyle{remexample}{10pt}{10pt}{}{}{\bfseries}{.}{10pt}{{\thmname{#1}\thmnumber{ #2}\thmnote{ (#3)}}}
\newtheoremstyle{ass}{10pt}{10pt}{}{}{\bfseries}{.}{10pt}{{\thmname{#1}\thmnumber{ A#2}\thmnote{ (#3)}}}

\theoremstyle{thmlemcorr}
\newtheorem{theorem}{Theorem}
\numberwithin{theorem}{section}
\newtheorem{lemma}[theorem]{Lemma}
\newtheorem{corollary}[theorem]{Corollary}

\theoremstyle{thmlemcorr*}
\newtheorem{theorem*}{Theorem}
\newtheorem{lemma*}[theorem]{Lemma}
\newtheorem{corollary*}[theorem]{Corollary}
\newtheorem{proposition*}[theorem]{Proposition}
\newtheorem{problem*}[theorem]{Problem}
\newtheorem{conjecture*}[theorem]{Conjecture}
\newtheorem{definition*}[theorem]{Definition}

\theoremstyle{remexample}
\newtheorem{remark}[theorem]{Remark}

\theoremstyle{ass}

\def\Xint#1{\mathchoice
{\XXint\displaystyle\textstyle{#1}}
{\XXint\textstyle\scriptstyle{#1}}
{\XXint\scriptstyle\scriptscriptstyle{#1}}
{\XXint\scriptscriptstyle\scriptscriptstyle{#1}}
\!\int}
\def\XXint#1#2#3{{\setbox0=\hbox{$#1{#2#3}{\int}$ }
\vcenter{\hbox{$#2#3$ }}\kern-.56\wd0}}
\def\intavg{\Xint-}

\DeclareMathOperator{\cof}{cof}

\newcommand{\Ocal}{\mathcal{O}}

\newcommand{\Fbb}{\mathbb{F}}

\newcommand{\Ibb}{\mathbb{I}}

\newcommand{\Mbb}{\mathbb{M}}

\newcommand{\Rbb}{\mathbb{R}}

\newcommand{\Tbb}{\mathbb{T}}
\newcommand{\T}{\mathbb{T}}

\DeclareMathOperator{\diverg}{div}
\DeclareMathOperator{\Div}{div}

\newcommand{\del}{\partial}

\newcommand{\norm}[1]{\|#1\|}

\newcommand{\R}{\mathbb{R}}

\def \la {\langle}
\def \ra {\rangle}

\newcommand{\vr}{\varrho}

\newcommand{\dx}{ dx}

\def\Xint#1{\mathchoice
{\XXint\displaystyle\textstyle{#1}}%
{\XXint\textstyle\scriptstyle{#1}}%
{\XXint\scriptstyle\scriptscriptstyle{#1}}%
{\XXint\scriptscriptstyle\scriptscriptstyle{#1}}%
\!\int}
\def\XXint#1#2#3{{\setbox0=\hbox{$#1{#2#3}{\int}$}
\vcenter{\hbox{$#2#3$}}\kern-.5\wd0}}

\renewcommand{\epsilon}{\varepsilon}
\renewcommand{\phi}{\varphi}

\begin{document}


\title[Onsager's Conjecture for Conservation of Energy and Entropy]{
Onsager's conjecture in bounded domains for the conservation of  entropy and other companion laws
}

\author{C. Bardos}
\address{\textit{Claude Bardos:  }Laboratoire J.-L. Lions, BP187, 75252 Paris Cedex 05, France. Email:
}
\email{claude.bardos@gmail.com}

\author{P. Gwiazda}
\address{\textit{Piotr Gwiazda:} Institute of Mathematics, Polish Academy of Sciences, \'Sniadeckich 8, 00-656 Warszawa, Poland}
\email{pgwiazda@mimuw.edu.pl}

\author{A. \'Swierczewska-Gwiazda}
\address{\textit{Agnieszka \'Swierczewska-Gwiazda:}  Institute of Applied Mathematics and Mechanics, University of Warsaw, Banacha 2, 02-097 Warszawa, Poland}
\email{aswiercz@mimuw.edu.pl}

\author{E.S. Titi}
\address{\textit{Edriss S. Titi:} Department of Mathematics,
                 Texas A\&M University, 3368 TAMU,
                 College Station, TX 77843-3368, USA. Department of Applied Mathematics and Theoretical Physics, University of Cambridge,
Wilberforce Road, Cambridge CB3 0WA, UK. Department of Computer Science and Applied Mathematics, The Weizmann Institute of Science, Rehovot 76100, Israel.}
                 \email{titi@math.tamu.edu \, and \, edriss.titi@weizmann.ac.il}

\author{E. Wiedemann}

\address{\textit{Emil Wiedemann:} Institute of Applied Analysis, Universit\"at Ulm, Helmholtzstr.\ 18, 89081 Ulm, Germany}
\email{emil.wiedemann@uni-ulm.de}

\begin{abstract}
 We show that weak solutions of  general conservation laws in bounded domains conserve their generalized entropy, and other respective companion laws, if they possess a certain fractional differentiability of order 1/3 in the interior of the domain, and if the normal component of the corresponding  fluxes tend to zero as one  approaches the boundary. This extends various recent results of the authors.
\end{abstract}





\noindent\textsc{Date:} February 12, 2019



\maketitle



 \tableofcontents


\section{Introduction}


We consider very general systems of conservation laws of the form
\begin{equation}\label{conslawintro}
\diverg_X(G(U(X)))=0\quad\text{for $X\in\mathcal{X}$},
\end{equation}
 where $\mathcal{X}\subset \R^{k+1}$ is open, $U:\mathcal{X}\to\mathcal O$ for some subset $\mathcal{O}\subset \R^n$, and $G:\mathcal{O}\to\R^{n\times(k+1)}$. It is shown in Section~\ref{applications} below that many important evolution equations of hyperbolic character can be written in this {general}  form, including the incompressible and compressible Euler systems, the equations of inviscid magnetohydrodynamics, and the equations of elastodynamics.

{Many systems of the form  (\ref{conslawintro}) come} with so-called \emph{companion laws} (see~\eqref{complaw} below), according to which sufficiently regular solutions satisfy one or several (sometimes infinitely many) additional conservation laws. Oftentimes, these companion laws can be interpreted as the conservation of energy or entropy. In particular, in the context of hyperbolic conservation laws, the notion of {\emph{ generalized entropy solution}} refers to these additional formally conserved quantities.

Of course a physical entropy can not, in general, be viewed as a conserved quantity; quite the opposite, it is (with the mathematical sign convention) typically \emph{dissipated}, i.e.\ it decreases in time.

In all examples of dissipation of energy or entropy, a certain degree of irregularity is required to violate the corresponding companion law. This is true for Scheffer's solutions and subsequent refinements in the case of the incompressible Euler equations, but also for hyperbolic conservation laws, where the mechanism of entropy dissipation by shock formation is classically known. Mathematically, the formal conservation of energy/entropy relies on the chain rule, which may not be valid for non-Lipschitz functions.

The question thus arises what is the threshold regularity of the solutions above which companion laws are guaranteed to hold.
 {In 1949 Onsager \cite{ON} related this issue to the Kolmogorov statistical theory of turbulence and proposed (what then became known as the Onsager conjecture) that in $3d$ for the solutions of the incompressible Euler equations this threshold should be H\"older regularity with critical exponent $\alpha=\frac1 3 $. These recent years have seen definite progress toward the resolution of this conjecture.}

 { On the one hand after the forerunner contributions of Scheffer \cite{Scheffer} (1993) and Shnirelman \cite{shnirel1} (2000),  with the introduction,  by  C.\ De Lellis,  and L.\ Sz\'ekelyhidi, of the tools of  {\it convex integration}, constant progress have been made  in particular with the contributions of Isett, and of Buckmaster, De Lellis, Sz\'ekelyhidi, and Vicol \cite{isett16, buckmasteretal17}.  What has been shown there is the following: Given any energy profile $e(t)$, and any $\alpha<\frac13\,,$  there exist space periodic solutions of the $3d$ incompressible Euler equations which belong to $C^\alpha(\T^3\times (0,T))$ and which satisfy the non conservative energy relation:
 $$
 \int_{\mathbf T^3} |u(x,t)|^2dx =e(t)\,.
 $$}
{On the other hand the first proof  of a sufficient  $\alpha > \frac13 $ regularity condition for the conservation of energy of weak solutions to the $3d$ Euler equations  in the full space or subject to periodic boundary conditions goes back to  1994, cf.\  \cite{CET} (and \cite{GEY}, for the case when $\alpha > \frac12$). New refinements and extensions of these results to other systems are the object of the present contribution (see also \cite{BTWP218}).
First, for problems defined in an open set $\mathcal{ X } \subset \R^{k+1} $, a refined Besov-BMO type space, introduced in \cite{FjWi} and denoted here by $\underline{B}_{3,\textit{VBMO}}^{1/3}(\mathcal{X})$ (cf.~\eqref{VMOcondition-1}), is used, for which, with $\alpha>\frac13$, one has the following inclusions:}
 {\begin{equation}
  C^\alpha\subset B_{3,\infty}^{\alpha}\subset B_{3,c_0}^{1/3}\subset  \underline{B}_{3,\textit{VMO}}^{1/3}
  \subset B_{3,\infty}^{1/3} \,.\label{cetshvydkoy}
  \end{equation} }
    { For the $3d$ Euler equations, conservation of energy was proven for solutions belonging to $B_{3,\infty}^{\alpha}$ by Constantin, E, and Titi  \cite{CET}. Then this was extended to solutions belonging to $B_{3,c_0}^{1/3}$  by \cite{CCFS08},
  where it was also shown that this result is almost optimal because one can construct divergence free vector fields $U\in    B_{3,\infty}^{\frac 1 3}$ with a non zero energy flux. The conservation result of \cite{CCFS08} was recently improved to $\underline{B}_{3,\textit{VMO}}^{1/3}$ in~\cite{FjWi}. Hence  $ \underline{B}_{3,\textit{VMO}}^{1/3}$ appears to be an almost optimal regularity class for the conservation of energy.
Moreover, the functions $ U\in \underline{B}_{3,\textit{VMO}}^{1/3}$ are characterized by a simple property in the physical space  which makes this space well adapted to localized formulation of an extra conservation law. This makes this space a good tool to deal with the case of domains with boundary, extending the results of \cite{BTW18, BTWP218, DN18, GwMiSw} and in particular relaxing the H\"older $\alpha>\frac13$ regularity hypothesis. At the end of the day the use of the Besov-$VMO$ space leads to a very concise proof of our main theorem  (see formulas \eqref{p2} and \eqref{p1} in the proof of Theorem \ref{localthm}). }



{We show in Section~\ref{localtoglobal} that similar conditions as discovered in~\cite{BTW18, DN18} also guarantee validity of companion laws for general global conservation laws of the type~\eqref{conslawintro} (see Theorem~\ref{globalthm} below). That is, if the $ \underline{B}_{3,\textit{VMO}}^{1/3}$ condition~\eqref{VMOcondition-1} is satisfied in the interior of the domain (but not necessarily uniformly up to the boundary), and if the normal component of the flux tends to zero suitably as the boundary is approached, then the corresponding global companion law is satisfied.} The study of such more general nonlinearities was initiated in~\cite{FGSW}, where the isentropic compressible Euler equations (for which the density appears in a non-quadratic way) were studied, but only in the absence of physical boundaries. The general framework for conservation laws of the type~\eqref{conslawintro} was introduced in~\cite{GwMiSw}, and considered for bounded domains (but not in the optimal functional setting) in~\cite{BTWP218}.

{
In the final Section~\ref{applications}, we show how our general results can be applied to various important physical systems from fluid and solid mechanics. In particular, we demonstrate that our boundary assumption~\eqref{boundaryass} relates to the natural boundary conditions usually imposed on the respective equations, e.g.\ the impermeability (or slip) condition for inviscid fluids or the zero traction boundary condition for elastic solids. It is noteworthy that in these examples, the boundary conditions shown to ensure entropy/energy conservation are those that render the respective equations locally well-posed in classes of inital data with sufficient smoothness. }

The framework of this paper is, as explained, very general, but the $C^2$ assumption in Theorem~\ref{localthm} excludes some interesting degenerate situations, like the compressible Euler system with possible vacuum. Such degeneracies were dealt with in~\cite{AkWi, AkDeSkWi}, where it turned out that the analysis beyond $C^2$ nonlinearities is very delicate and can not be expected to be carried out in the generality of~\eqref{conslawintro}.

As a final remark, as emphasized above the space $\underline{B}_{3,\textit{VMO}}^{1/3}$ gives a universal frame for a sufficient condition for the validity of the extra conservation law. For incompressible models, however, it is only in the case of the $3d$ incompressible Euler equations that such a condition has been shown to be (almost) necessary. For genuinely nonlinear hyperbolic conservation laws, on the other hand, the optimality of the exponent $1/3$ is easily obtained from shock solutions, at least on the scale of $L^3$-based spaces: As observed in~\cite{CCFS08, FGSW}, the space $BV\cap L^\infty$, which contains shocks, embeds into the Besov space $B_{3,\infty}^{1/3}$, which is critical in our results.


\section{Extension and Adapted Function Spaces}

{
We formulate a Besov-type condition stated in~\cite{FjWi} in a local version. For all $ {\mathcal{X}}'\subset\subset {\mathcal{X}}\subset \R^{k+1}$  and $0<\epsilon<\frac{d({\mathcal{X}}',\partial {\mathcal{X}})}{2}$ let
\begin{equation}\label{VMOcondition-1}
\int_{\mathcal{X}'}\intavg_{B_\epsilon(X)}|U(X)-U(Y)|^3dYdX\le\omega_{{\mathcal{X}}'}(\epsilon)\epsilon
\end{equation}
where $\omega_{{\mathcal{X}}'}(\epsilon)>0$ is a nonnegative function on ${\mathcal X}$ which tends to zero as $\epsilon$ tends to zero. We write $U\in \underline{B}_{3,\textit{VMO}}^{1/3}(\mathcal{X})$ if $U \in L^3(\mathcal X)$ and it satisfies condition \eqref{VMOcondition-1}.
}

 As explained in~\cite{FjWi}, this condition is more general than the critical Besov condition $U\in B_{3,c_0}^{1/3}(\mathcal{X})$ from~\cite{CCFS08}, as the latter reads
 \begin{equation}\label{aver-cond}
\lim_{Z\to 0}\frac{1}{|Z|}\int_{\mathcal{X}}
|U(X)-U(X+Z)|^3dX=0,
\end{equation}
and it is easy to see that this implies~\eqref{VMOcondition-1}.
{
Indeed, rewrite~\eqref{aver-cond} as follows:
 \begin{equation*}
\frac{1}{|Z|}\int_{\mathcal{X}}
|U(X)-U(X+Z)|^3dX\le \omega(Z)
\end{equation*}
with $\omega(Z)$ converging to zero as $Z\to0$. Fix $0<\epsilon<1$, then obviously for all $|Z|\le \epsilon$ also
 \begin{equation*}
\frac{1}{\epsilon}\int_{\mathcal{X}}
|U(X)-U(X+Z)|^3dX\le \omega(Z)
\end{equation*}
holds and we can integrate
 \begin{equation*}
\frac{1}{\epsilon}\intavg_{B_\varepsilon(0)}\int_{\mathcal{X}}
|U(X)-U(X+Z)|^3dXdZ\le \intavg_{B_\varepsilon(0)}\omega(Z)dZ.
\end{equation*}
Let us now define $\overline\omega(\varepsilon):= \intavg_{B_\varepsilon(0)}\omega(Z)dZ$.
It is easy to verify that $\overline\omega(\varepsilon)$ vanishes as $\varepsilon\to0$.
Finally, using Fubini's theorem and a change of variables, we arrive at condition \eqref{VMOcondition-1}.

}


{Localization proofs, as they are used in this paper, involve the action of $U$, or of a function of $U$, on a given test function $\psi\in \mathcal D(\mathcal X)$. If the support of $\psi$ is strictly contained in an open set $\mathcal X'\subset \subset \mathcal X$, and if $0<\epsilon <\epsilon_0$ is chosen small enough, one may choose additional open sets $\mathcal X_1$, $\mathcal X_2$ satisfying the inclusions
\begin{equation}
\operatorname{supp}(\psi) \subset\subset \mathcal X'\subset \subset \mathcal X_2\subset \subset\mathcal X_1\subset \subset\mathcal X \label{multisupport}
\end{equation}
with $\mathcal X_2$ containing an $\epsilon$-neighborhood $\mathcal X'_\epsilon$ of $\mathcal X'$ and $\mathcal X_1$ containing an $\epsilon$-neighborhood $\mathcal X_2^\epsilon $ of $\mathcal X_2\,.$ Then, proceeding as in \cite{BTWP218} Section 2.1 (see also~\cite{BTW18}),   introducing a function $I\in   \mathcal D(\R^{k+1})$ with support in $\mathcal X_1$, equal to $1$ in $\mathcal X_2$, and with gradient supported in $ \mathcal X_1\backslash \overline{\mathcal X_2^\epsilon}$, we define (and denote by $[T]$) the extension to $\mathcal D'(\R^{k+1} )$ of any distribution $T\in\mathcal D'(\mathcal X)$ by the formula:
\begin{equation}
\la [T], \psi\ra=\la T, I\psi\ra\,.\label{extension}
\end{equation}
This extension allows us, in particular, to make sense of mollifications of a given function on $\mathcal X$, when tested against $\psi\in \mathcal D(\mathcal X)$.

 As in previous contributions a sequence of mollifiers will be used. They are  defined as follows:  }

Starting from a positive function $s\mapsto \eta(s) \in \mathcal D(\R^{k+1})$ with support in $|s|<1$ and total mass $\int\eta(s)ds=1$, we denote by $\eta_\epsilon(X)$ the function
 \begin{equation}
 \eta_\epsilon(X)= \frac1{\epsilon^{k+1}}\eta\left(\frac{X}{\epsilon}\right)
\end{equation}
 and use the notation $S_\epsilon$  for the mollification $\eta_\epsilon\star S $ of a distribution $S
\in{\mathcal D}'(\R^{k+1})\,.$
 General   results proven below under the hypothesis (\ref {VMOcondition-1}) are almost direct  consequences of the following lemmas:
\begin{lemma}\label{VBMOlemma}
Let $\mathcal{X}'\subset\subset\mathcal{X}$ and $\epsilon>0$ so small that $\mathcal{X}$ contains an $\epsilon$-neighbourhood of $\mathcal{X}'$ leading to the construction (\ref{multisupport}) and (\ref{extension}). Let $U\in \underline{B}_{3,\textit{VMO}}^{1/3}(\mathcal{X})$, then, for some $\omega_{(U,\mathcal X')}:(0,1)\to\R^+$ such that $\liminf_{\epsilon\to0}\omega_{(U,\mathcal X')}(\epsilon)=0$, one has
\begin{equation}
\norm{D_X[U]_\epsilon}_{L^3(\mathcal{X}')}\leq  ( \omega _{(U,\mathcal X')} (\epsilon))^{\frac 13}\epsilon^{-2/3}.
\end{equation}
\end{lemma}
\begin{proof}
By  Jensen's inequality, with
\begin{equation}
D_X\eta_\epsilon(X)=
\frac 1\epsilon
  \frac1{\epsilon^{k+1}}
 (D_X\eta)\left(\frac X \epsilon\right) \quad \hbox{and} \quad   C= \left(\int_{\R^{k+1}}|(D_X(\eta (X)|dX\right)^2
\end{equation}
and one standard integration by parts we have:
\begin{equation}
\begin{aligned}
\norm{D_X[U]_\epsilon}_{L^3(\mathcal{X}')}&=\left(\int_{\mathcal{X}'}\left|\int_{B_\epsilon(X)}U(Y)D_X\eta_\epsilon(X-Y)dY\right|^3dX\right)^{1/3}\\
&=\left(\int_{\mathcal{X}'}\left|\int_{B_\epsilon(X)}(U(Y)-U(X))D_Y\eta_\epsilon(X-Y)dY\right|^3dX\right)^{1/3}\\
&\leq C\epsilon^{1/3}\epsilon^{-1}\left(\frac{1}{\epsilon}\int_{\mathcal{X}'}\intavg_{B_\epsilon(0)}\left|U(Y)-U(X-Y)\right|^3dYdX\right)^{1/3}\\
&=( \omega _{(U,\mathcal X')} (\epsilon))^{\frac 13} \epsilon^{-2/3}.
\end{aligned}
\end{equation}
\end{proof}
\begin{corollary} Under the above hypothesis, with $B \in W^{1,\infty} (\mathcal O; \R^n)$  and $\psi \in \mathcal D(\mathcal X) $ one has:
\begin{equation}
\|D_X( (B([U]_\epsilon )(.)^T   \psi(.))\|_{L^3}\le  C(B,\psi) ( \omega _{(U,\mathcal X')} (\epsilon))^{\frac 13} \epsilon^{-2/3}. \label{basic1}
\end{equation}
\end{corollary}
\begin{proof}
Write
$$D_X( (B([U]_\epsilon )(.)^T   \psi(.))= D_U B([U]_\epsilon )D_X[U]_\epsilon \psi + B([U]_\epsilon )(.)^T  D_X\psi $$
and apply the above estimates.
\end{proof}

\begin{lemma}\label{commestimate}
Let $\mathcal{X}'\subset\subset\mathcal{X}$ and $\epsilon>0$ so small that $\mathcal{X}$ contains an $\epsilon$-neighbourhood of $\mathcal{X}'$. Assume that $\mathcal{O}\subset\R^n$ is convex and $G\in C^2(\overline{\mathcal{O}};\R^{n\times(k+1)})$. Let $U\in \underline{B}_{3,\textit{VMO}}^{1/3}(\mathcal{X})$, then
\begin{equation}
\norm{[G(U)]_\epsilon-G([U]_\epsilon)}_{L^{3/2}(\mathcal{X}')}\leq  C \left(\int_{\mathcal{X}'}\intavg_{B_\epsilon(X)}|U(X)-U(Y)|^3dYdX\right)^{2/3}, \label{basic2}
\end{equation}
where $C$ depends only on $\eta$, the dimension of $\mathcal{O}$, and $ \|D_U^2G\|_{L^\infty(\mathcal O)}$ (but not on $U$ or $\epsilon$).
\end{lemma}
\begin{proof}
With minor improvements we follow the proofs of  Lemma 3.1 of \cite{GwMiSw} and of Lemma 2.3 of \cite{BTWP218}. Starting from the pointwise estimate (16) in Section 2 of \cite {BTWP218}, we immediately obtain for $X\in \mathcal X'$
\begin{equation}
\begin{aligned}
&(|[G(U)]_\epsilon(X)-G([U]_\epsilon)(X))|)^{\frac 32}\\
&\quad \quad \quad \quad \leq C \Bigg(\int_{\R^{k+1}} |U(X-Y) -U (X)|^2\eta_\epsilon(Y)dY\Bigg)^{\frac 32}\,.
\end{aligned}
\end{equation}
for a constant as stated. Then with the H\"older inequality one has
\begin{equation}
\begin{aligned}
&\int_{\mathcal X'} dX\Bigg(\int_{\R^{k+1}} |U(X-Y) -U (X)|^2\eta_\epsilon(Y)dY\Bigg)^{\frac 32} \\
&\le  \int_{\mathcal{X}'}dX \int_{\R^{k+1}} |U(X-Y) -U (X)|^3\eta_\epsilon(Y)dY
 \end{aligned}
 \end{equation}
 and therefore:
 \begin{equation}
 \begin{aligned}
& \Bigg( \int_{\mathcal X'} (|[G(U)]_\epsilon(X)-G([U]_\epsilon)(X))|)^{\frac 32}dX\Bigg)^{\frac 2 3}\\
&\le  C\Bigg( \int_{\mathcal{X}'}dX \int_{\R^{k+1}} |U(X-Y) -U (X)|^3\eta_\epsilon(Y)dY )\Bigg)^{\frac 23}\\
&\le  C\left(\int_{\mathcal{X}'}\intavg_{B_\epsilon(X)}|U(X)-U(Y)|^3dYdX\right)^{2/3}\,.\label{basic3}
 \end{aligned}
 \end{equation}
\end{proof}
\begin{remark} \label{r42} Since in the formulas  (\ref{basic2})  and then (\ref{basic3}) only the second derivative of $U\mapsto G(U)$ appears, such formulas are trivial when this function is affine. Therefore the
Corollaries 4.1 -- 4.3 of~\cite{GwMiSw} transfer directly to the present situation giving the following results which will be used in Sections \ref{applications}.
\begin{itemize}
\item If $G=(G_1,\ldots,G_s,G_{s+1},\ldots,G_k)$ for \emph{affine} functions $G_1,\ldots,G_s$, and $\mathcal{X}=\mathcal{Y}\times\mathcal{Z}$ for some $\mathcal{Y}\subset\R^s$ and $\mathcal{Z}\subset\R^{k+1-s}$, then in Theorem~\ref{localthm} below it suffices to assume
\begin{equation*}
\liminf_{\epsilon\to0}\frac{1}{\epsilon}\int_{\mathcal{Y}_1}\int_{\mathcal{Z}_1}\intavg_{B_\epsilon(Z)\cap\mathcal{Z}_1}|U(X,Y)-U(X,Z)|^3dYdZdX=0
\end{equation*}
for all $\mathcal{Y}_1\subset\subset\mathcal{Y}$, $\mathcal{Z}_1\subset\subset\mathcal{Z}$. One should keep in mind the situation $\mathcal{Y}=(0,T)$, $\mathcal{Z}=\Omega$, and $A=\operatorname{id}$ in the terminology of the next section.
\item If $U=(V_1,V_2)$ (with $V_1=(U_1,\ldots,U_s)$, $V_2=(U_{s+1},\ldots,U_n)$), if $B$ is independent of $V_1$, if $G=G_1(V_1)+G_2(V_2)$, and if $G_1$ is linear, then for $U_1,\ldots U_s$ it suffices to assume $U_1,\ldots U_s\in L^3_{loc}(\mathcal{X})$ in Theorem~\ref{localthm}.
\item If the $j$-th row of $G$ is affine, then Theorem~\ref{localthm} remains true even if $B_j$ is only locally Lipschitz in $\mathcal{O}$.
\end{itemize}
\end{remark}
\section{Companion Laws at Critical Regularity}\label{companion}
%
To apply the previous lemma   to     systems of (not necessarily hyperbolic) conservation laws in the full generality of \cite{GwMiSw}, we consider the problem:
\begin{equation}\label{conslaw}
\diverg_X(G(U(X)))=0\quad\text{for $X\in\mathcal{X}$.}
\end{equation}
As in \cite{GwMiSw}, we assume $\mathcal{X}\subset \R^{k+1}$ is open, $U:\mathcal{X}\to\mathcal O$ for some subset $\mathcal{O}\subset \R^n$, and $G:\mathcal{O}\to\R^{n\times(k+1)}$. In all our applications in Section~\ref{applications} below, $X=(x,t)$ will be interpreted as a point of space-time.

Following \cite{Daf}, we will consider \emph{companion laws} of the form
\begin{equation}\label{complaw}
\diverg_X(Q(U(X)))=0\quad\textit{for $X\in\mathcal{X}$,}
\end{equation}
where $Q:\mathcal{O}\to\R^{k+1}$ is a smooth function such that there exists another smooth function $B:\mathcal{O}\to\R^{n}$ satisfying the relation
\begin{equation}\label{Qdef}
D_UQ_j(U)=B(U)D_UG_j(U)\quad\text{for all $U\in\mathcal{O}$, $j\in\{0,\ldots,k+1\}$.}
\end{equation}
Note that in the context of hyperbolic conservation laws, the definition of $Q$ corresponds to the well-known notion of \emph{entropy--entropy-flux pairs}, whereas the companion law \eqref{complaw} can be interpreted as the usual \emph{entropy equality}. The companion law is seen to be true by virtue of the chain rule as long as the latter is applicable, e.g.\ for a solution $U$ of~\eqref{conslaw} which is in $C^1$ (or Lipschitz). However, the companion law may fail to be true for \emph{weak solutions}, i.e.\ vector fields $U$ that satisfy

\begin{equation}
\int_{\mathcal{X}}G(U(X)):D_X\psi(X)dX=0
\end{equation}
for every $\psi\in C_c^1(\mathcal{X};\R^n)$. Note carefully that the definition of weak solution is purely \emph{local} in the sense that $U$ is a weak solution on $\mathcal{X}$ if and only if it is a weak solution on every subset $\mathcal{X}'\subset\subset\mathcal{X}$. In particular, no boundary condition is included in this formulation. A weak solution of the companion law \eqref{complaw} is defined analogously.

Then we have the following improvements  of Theorem 1.1 in \cite{GwMiSw} and Theorem 2.1 in \cite{BTWP218}:

{\begin{theorem}\label{localthm}
Assume that $\mathcal{O}\subset\R^n$ is convex, $G\in C^2(\overline{\mathcal{O}};\R^{n\times(k+1)})$, $Q\in C^1(\mathcal{O};\R^{k+1})$  and $B\in C^{1}(\mathcal{O};\R^{ n})$ and the following conditions hold:
 \begin{equation*}
        \label{eq:assumpt_convex}
         \begin{aligned}
                D_U{B}\in L^{\infty}(\Ocal;\R^{ n}), \quad|B(V)|\le C(1+|V|)
         \\
         |Q(V)|\leq C(1+|V|^3)\   \mbox{for all $V\in\Ocal$},
         \\
         \sup_{i,j \in{1,\dots,d}}\|\partial_{U_i}\partial  _{U_j} G(U)\|_{L^\infty(\Ocal;\,\Mbb^{n\times (k+1)})}<+\infty
         \end{aligned}
    \end{equation*}
for some constant $C$ independent of $V$, and \eqref{Qdef} holds.
If $U$ is a weak solution of \eqref{conslaw} such that $U\in \underline{B}_{3,\textit{VMO}}^{1/3}(\mathcal{X}_1)$ for every $\mathcal{X}_1\subset\subset \mathcal{X}$, then $U$ is also a weak solution of the companion law~\eqref{complaw}.
\end{theorem}}
{\begin{proof}
Following the definition of derivative in the sense of distributions, we consider a test function  $\psi\in \mathcal D (\mathcal{X})$ supported in $\mathcal{X}'\subset\subset\mathcal{X}$. Then using the construction described by the formulas   (\ref{multisupport}) and (\ref{extension} ) and Lebesgue's dominated convergence theorem, first we write:
\begin{equation}
-\la \diverg_X Q(U), \psi \ra = \lim_{\epsilon \rightarrow 0} \int_{\R^{k+1}} Q([U]_\epsilon)\cdot D_X\psi dX. \label{p1}
\end{equation}
Then since $[U]_\epsilon \in \mathcal D(\R^{k+1})$ one has:
\begin{equation}
\begin{aligned}
& \int_{\R^{k+1}} Q([U]_\epsilon),D_X\psi dX = -\int_{\R^{k+1}}\diverg Q([U]_\epsilon) \cdot \psi(X) dX \\
 &= -\int_{\R^{k+1}}B([U]_\epsilon ) D_UG([U]_\epsilon) D_X ([U]_\epsilon) \cdot  \psi(X) dX\\
 &=-\int_{\R^{k+1}} D_UG([U]_\epsilon)D_X ([U]_\epsilon)  \cdot  (B([U]_\epsilon )(X)^T   \psi(X) dX\\
 &=- \int_{\R^{k+1}} D_X(G([U]_\epsilon)) \cdot  (B([U]_\epsilon )(X)^T   \psi(X) dX\\
 &=\int_{\R^{k+1}} (G([U]_\epsilon) -G([U]))  \cdot D_X( (B([U]_\epsilon )(X)^T   \psi(X)) dX\\
 &+ \la G([U]),  D_X( (B([U]_\epsilon )(X)^T   \psi(X)) \ra.
\label{p2}
\end{aligned}
\end{equation}
 Since   $\diverg_X G(U) =0$ in $\mathcal D'(\mathcal X)$, the last term of (\ref{p2}) is equal to $0\,.$ Then from (\ref{p1}) and (\ref{p2}) with (\ref{basic1}) and (\ref{basic2}) we have
\begin{equation}
\begin{aligned}
&|\la \diverg_X Q(U), \psi \ra| \le \lim_{\epsilon\rightarrow 0} \left|\int_{\R^{k+1}} (G([U]_\epsilon) -G([U]))  \cdot D_X( (B([U]_\epsilon )(X)^T   \psi(X)) dX\right|\\
&\le \|(G([U]_\epsilon) -G(U))\|_{L^{\frac32}(\mathcal{X}')}\|D_X( (B([U]_\epsilon ) ^T   \psi )\|_{L^{3}} \le \omega_{(U,\mathcal X)}(\epsilon)\,. \end{aligned}
\end{equation}
Hence in $\mathcal D'(\mathcal X)$ one has
\begin{equation}
  \diverg_X (Q(U) ) =0\,.
\end{equation}
\end{proof}}

\section{From Local to Global Companion Laws}\label{localtoglobal}
We now specialise to the case where $\mathcal{X}=\Omega\times(0,T)$ for some domain $\Omega\subset\R^k$, and we write $X=(x,t)$. Then $G$ can be written in the form
\begin{equation}\label{FA}
G(U)=(F(U), A(U))
\end{equation}
for some $A:\mathcal{O}\to \R^n$ and $F:\mathcal{O}\to\R^{n\times k}$, so that the conservation law \eqref{conslaw} reads as
\begin{equation}
\partial_t (A(U(x,t)))+\diverg_x F(U(x,t))=0,
\end{equation}
or, in weak formulation,
\begin{equation}\label{timeweak}
\int_0^T\int_\Omega \partial_t\psi(x,t) \cdot A(U(x,t))+\nabla_x\psi(x,t): F(U(x,t))=0
\end{equation}
for any $\psi\in C_c^1(\Omega\times(0,T);\R^n)$.

Setting $Q(U)=(q(U),\eta(U))$ for $q:\mathcal{O}\to \R^{k}$ and $\eta:\mathcal{O}\to \R$, we accordingly consider companion laws of the form
\begin{equation}\label{timecompanion}
\partial_t (\eta(U(x,t))) +\diverg_x q(U(x,t))=0,
\end{equation}
where $\eta$ and $q$ satisfy
\begin{equation}\label{compat}
\begin{aligned}
D_U\eta(U)&=B(U)D_U A(U),\\
D_U q_j(U)&=B(U)D_U F_j(U)\quad\text{for $j=1,\ldots,k$}
\end{aligned}
\end{equation}
for some smooth map $B:\mathcal{O}\to \R^{n}$.

{In the following, we assume that  $\Omega\subset \R^k$ is an open set with a bounded Lipschitz boundary $\del \Omega$, and therefore the exterior normal to the boundary denoted by $n(x)$ is defined almost everywhere. We denote by $d(x,\del\Omega)$ the distance of a point $x\in \Omega$ to $\del \Omega$.
Then  we observe the existence of a (small enough) $\epsilon_0$ with the following properties:
For   $ d(x,\del \Omega) \le \epsilon_0$ the function $x\mapsto d(x,\del\Omega)$ belongs to $W^{1,\infty}(\Omega)$ and there exists, for almost every such $x$, a unique point  $\hat x =\sigma(x) \in \del\Omega $ such that:
\begin{equation}
d(x,\del \Omega)<\epsilon_0 \Rightarrow d(x,\del\Omega) = |x-\sigma (x)| \, \quad \hbox{and} \quad \nabla_x d(x,\del\Omega) =-n(\sigma (x))\,.
\end{equation}}

Choosing a test function of the form $\psi(x,t)=\chi(t)\phi(x)$ for a generic $\chi$ and an approximation $\phi$ of the indicator function of $\Omega$, we can then impose similar conditions on the boundary behavior of the fluxes and use similar arguments to those developed in~\cite{BTW18}, to pass  from the local statement of Theorem~\ref{localthm} to a global one:

\begin{theorem}\label{globalthm}
Let $\mathcal{O}\subset \R^n$ convex, $A\in C^2(\overline{\mathcal{O}};\R^n)$, and $F\in C^2(\overline{\mathcal{O}};\R^{n\times k})$. Assume there exist $\eta\in C^1(\mathcal{O};\R)$, $q\in C^1(\mathcal{O};\R^{ k})$, and $B\in C^{1}(\mathcal{O};\R^{ n})$ such that

 \begin{equation*}
        \label{eq:assumpt_convex}
         \begin{aligned}
                D_U{B}\in L^{\infty}(\Ocal;\R^{ n}), \quad|B(V)|\le C(1+|V|)
         \\
         \sup_{i,j \in{1,\dots,d}}\|\partial_{U_i}\partial  _{U_j} A(U)\|_{C(\Ocal;\,\Mbb^{n\times (k+1)})}<+\infty, \\\sup_{i,j \in{1,\dots,d}}\|\partial_{U_i}\partial  _{U_j} F(U)\|_{C(\Ocal;\,\Mbb^{n\times (k+1)})}<+\infty
         \end{aligned}
    \end{equation*}
and
\begin{equation}
|\eta(V)|+|q(V)|\leq C(1+|V|^3)
\label{locest}
\end{equation}
for all $V\in\mathcal{O}$ and for some constant $C$ independent of $V$, and such that~\eqref{compat} holds.

If $U$ is a weak solution of~\eqref{timeweak} such that $U\in \underline{B}^{1/3}_{3,VMO}(\Omega_1\times(0,T))$ for every $\Omega_1\subset\subset\Omega$,
{ i.e.\ if
\begin{equation}
\begin{aligned}
& U \in L^3(  \Omega_1\times(0,T))\,, \hbox{ and satisfies the estimate}\\
&\hbox{ with}\quad  \omega(U, \epsilon,\Omega_1,T)= \frac{1}{\epsilon}\int_0^T\int_{{\Omega_1}}\intavg_{B_\epsilon(X)\cap{\Omega_1}}|U(X)-U(Y)|^3dYdX\,,\\
&\liminf_{\epsilon\to0} \omega(U, \epsilon,\Omega_1,T )=0\,,
\end{aligned}
\end{equation}}
and if
\begin{equation}\label{boundaryass}
\liminf_{\epsilon\to0} \int_0^T\frac{1}{\epsilon}\int_{\frac{\epsilon}{4} \le d(x,\del\Omega) \le \frac{\epsilon}{2}} \big| q(U(x,t))n(\sigma(x)) \big |dxdt=0,
\end{equation}
then
\begin{equation}
\frac{d}{dt}\int_\Omega \eta(U(x,t))dx=0
\end{equation}
in the sense of distributions.
\end{theorem}
\begin{remark}\label{flutangent}
Assumption~\eqref{boundaryass} is satisfied, in particular, if $U$ is continuous near the boundary, and satisfies the boundary condition $q(U(x,t))n(x)=0$ on $\partial\Omega$.
\end{remark}

{\begin{proof}
Let $\chi\in C_c^1(0,T)$ and for $\epsilon<\epsilon_0$
\begin{equation}
\phi^\epsilon(x)=\phi\left(\frac{d(x,\partial\Omega)}{\epsilon}\right)
\end{equation}
for some nonnegative function $\phi\in C^1((0,\infty))$ such that $\phi\equiv0$ on $(0,\frac{1}{4}]$ and $\phi\equiv1$ on $(\frac{1}{2},\infty)$. Then, by Theorem~\ref{localthm}, the companion law~\eqref{timecompanion} holds in the sense of distributions, so that in particular
\begin{equation}
\int_0^T\int_\Omega \phi^\epsilon(x)\chi'(t)\eta(U(x,t))dxdt+\int_0^T\int_\Omega \chi(t)\nabla_x\phi^\epsilon(x)q(U(x,t))dxdt=0.
\end{equation}
For the first integral, notice that $\phi^\epsilon\to1$ as $\epsilon\to0$ pointwise in $\Omega$, so that the integral converges to
\begin{equation}
\int_0^T\int_\Omega \chi'(t)\eta(U(x,t))dxdt
\end{equation}
as $\epsilon\to0$. For the second integral, we observe that
\begin{equation}
\begin{aligned}
&\hbox{for every}\,\, x \in \Omega,\,\,  \hbox{such that}\,\, \frac{\epsilon}{4} \ge d(x,\del\Omega)\,\,  \hbox{or}\,\, d(x,\del\Omega)\ge \frac{\epsilon}{2}, \,\, \hbox{one has}\,\, \nabla_x\phi^\epsilon(x)=0; \\
&\hbox{and for}  \quad \frac{\epsilon}{4} \le d(x,\del\Omega) \le \frac{\epsilon}{2}\quad \hbox{ one has} \quad  \nabla_x\phi^\epsilon(x)=-\frac{1}{\epsilon}\phi'\left(\frac{d(x,\partial\Omega)}{\epsilon}\right)
n(\sigma( x)).
\end{aligned}
\end{equation}
Therefore we can estimate
\begin{equation}
\begin{aligned}
\left|\int_0^T\int_\Omega \chi(t)\nabla_x\phi^\epsilon(x)q(U(x,t))dxdt\right|&\leq C\int_0^T|\chi(t)|\frac{1}{\epsilon}\int_{\frac{\epsilon}{4} \le d(x,\del\Omega) \le \frac{\epsilon}{2}}|q(U(x,t))n(\sigma(x))|dxdt\to 0\label{ldt}
\end{aligned}
\end{equation}
along a subsequence $\epsilon_l\to0$, by virtue of assumption~\eqref{boundaryass}. In total, with the Lebesgue dominated convergence theorem applied to the right hand side of (\ref{ldt}) we obtain
\begin{equation}
\int_0^T\int_\Omega \chi'(t)\eta(U(x,t))dxdt=0,
\end{equation}
as claimed.
\end{proof}}

\section{Applications to Conservation of Energy/Entropy}\label{applications}
\subsection{Incompressible Euler system}\label{incompEul}
Consider the system
\begin{align}
 \label{eq:E1}
        \partial_t v + \Div_x(v\otimes v) + \nabla_x p &= 0,\\ \label{eq:E2}
          \Div_x v &= 0,
\end{align}
for an unknown vector field $v\colon \Omega\times[0,T] \to \R^n$
and scalar  $p\colon \Omega\times[0,T]  \to \R$.

For the variable  $U=(v,p)$ we have  $A(U)=(v,0)$, $F(U)=(v\otimes v+p\Ibb,v)$.
 The entropy (which here is the kinetic energy density) is given by $\eta(U) = \frac 12 |v|^2$, and the flux by $q(U)=\left(\frac{|v|^2}{2}+p\right)v$.
Hence, assuming the usual slip boundary condition ${v}\cdot n=0$ on $\partial\Omega$, we have $q(U)\cdot n=0$ on $\partial\Omega$.
{

The function $B$ has the form $B(U)=(v, p-\frac{1}{2}|v|^2)$. It obviously does not have linear growth, thus Theorems~\ref{localthm} and \ref{globalthm} cannot be directly applied. However, we observe that this problem concerns only the last component of the vector $B(U)=(B_1(U), B_2(U))$ with $B_2(U)=p-\frac{1}{2}|v|^2$.  Notice that the flux  $G(U)$ has the last row linear (or even zero for $A(U)$), thus there are no error terms there produced when mollification is applied.
For this reason we can prove the same statement without assuming linear growth of all the components of $B$.

Observe also that Remark~\ref{r42} allows to relax the condition on the pressure. Indeed, as $G(U)=(F(U),A(U))$ and $A(U)$ is an affine function, ${\mathcal X}=\Omega\times[0,T]$, then it is enough to require that
\begin{equation*}
\liminf_{\epsilon\to0}\frac{1}{\epsilon}\int_{I}\int_{\Omega'}\intavg_{B_\epsilon(y)}|U(t,x)-U(t,y)|^3dydx dt=0
\end{equation*}
for all $I\subset\subset [0,T]$, $\Omega'\subset\subset \Omega$, which corresponds to 
$U\in L^3(0,T;\underline B^{1/3}_{3, VMO}(\Omega))$.
Moreover, we can write $U=(v,p)$ and notice that $B$ is independent of $p$. In addition

\begin{equation*}
G=\left(\begin{array}{cc}v^T&v\otimes v\\0&v\end{array}\right)+\left(\begin{array}{cc}0
&p{\mathbb I}\\0&0\end{array}\right).
\end{equation*}

Thus it is enough to assume that $p\in L^3_{loc}({\mathcal X}).$ 
Taking this into account  note that our Theorem~\ref{globalthm} yields a similar result as the one in~\cite{BT18} or in ~ \cite{BTW18}.
{However in both articles the elliptic equation
\begin{equation}
-\Delta p=\Div\Div(v\otimes v) \label{laplace}
\end{equation}
  was used to relax even further the integrability assumption on the pressure. In \cite{BT18}, the global H\"older regularity of the pressure was deduced from the global H\"older regularity of the velocity $v$, while  in \cite{BTW18} first a local result was proven with a much weaker assumption on the pressure, $p\in L^{3/2}_tH^{-\beta}_x$ for some $\beta>0$,  which will guarantee the local regularity of the pressure, and then the derivation of the global energy conservation was done as above.}
  {
  For completeness, let us recall the corresponding result from~\cite{BTW18}:

\begin{theorem}[Th. 4.1 from~\cite{BTW18}]\label{CETloc}
Let $(v,p) \in L^q((0,T); L^2(\Omega  ))\times \mathcal D'( \Omega\times(0,T)  )$, for some $q\in [1,\infty]$, be  a weak solution of the Euler equations  satisfying the following hypotheses:
\begin{enumerate}
\item For some $\varepsilon_0>0$, small enough,
\begin{subequations}
\begin{equation}
p\in L^{3/2} ((0,T); H^{-\beta}(V_{\varepsilon_0}))\,, \quad \hbox{with}\quad  \beta <\infty \,,\label{forlebesque1}
\end{equation}
where $V_{\varepsilon_0}=\{ x\in \Omega\,: d(x,\partial \Omega)<\varepsilon_0\}$\,;
\item
\begin{equation}
 \lim_{\varepsilon \rightarrow 0} \int_0^T\frac{1}{\epsilon}\int_{\frac{\epsilon}{4} \le d(x,\del\Omega) \le \frac{\epsilon}{2}}\left|\left(\frac {|v|^2}2 + p\right) v(t,x)\cdot  n(\sigma( x))\right|\, dxdt =0\,; \label{forlebesgue2}
\end{equation}
\end{subequations}
\item For every open set $\tilde Q = \tilde\Omega \times(t_1,t_2)\subset\subset  \Omega\times(0,T) $ there exists $\alpha(\tilde Q)>1/3$ such that $v$~satisfies:
\begin{equation}
 \int_{t_1}^{t_2}  \|v(.,t)\|^3_{  C^{\alpha(\tilde Q)} (\overline{\tilde\Omega})}dt \le M(\tilde Q) <\infty \,. \label{local2}
\end{equation}
%
\end{enumerate}
Then, $(v,p)$ globally conserves the energy, i.e.,  for any $0 < t_1<t_2< T$ it satisfies the relation:
\begin{equation*}
\|v(t_2)\|_{L^2(\Omega)}=\|v(t_1)\|_{L^2(\Omega)}.
\end{equation*}
 Moreover, $v \in L^\infty((0,T); L^2(\Omega)) \cap C((0,T); L^2(\Omega))$.
\end{theorem}

  As observed in~\cite{BTW18} the hypothesis and conclusion of the above theorem, Theorem  \ref{CETloc}, are consistent with the situation where the behavior of the fluid in the vanishing viscosity limit is described by the Prandlt ansatz. This is also consistent with the $1/3-$Kolmogorov Law because in such situation the $\alpha>1/3$ regularity together with  condition (\ref{forlebesgue2}) imply the absence of anomalous energy dissipation. The results of~\cite{BTW18} have already been expanded in several direction in~\cite{DN18}. The authors of~\cite{DN18} use the
$B^{1/3}_{3,c_0}$ regularity,  cf. (\ref{cetshvydkoy}), instead of the  H\"older regularity. Moreover, they provide several avatars of the boundary condition (\ref{forlebesgue2}) which may be useful for the connection with the interpretation of this hypothesis in term of absence of ``turbulent boundary layer''.
}

Eventually let us remark that the theory presented here can also be applied to the \emph{inhomogeneous} incompressible Euler equations, where the density is no longer constant,

\begin{align}
 \label{eq:NE1}
 \begin{aligned}
\partial_t \rho + \Div_x(\rho v) &=0, \\ 
\partial_t (\rho v) + \Div_x(\rho v \otimes v) + \nabla_x p &= 0,\\ 
          \Div_x v &= 0,
          \end{aligned}
\end{align}
for an unknown vector field $v \colon \Omega\times[0,T] \to \R^n$ and scalar fields $\rho\colon \Omega\times[0,T] \to \R_+$ and $p\colon \Omega\times[0,T] \to \R$.
In this case, for the variable  $U=(\rho,v, p)$, we have  $A(U)=(\rho,\rho v,0)$ and $F(U)=(\rho v,\rho v\otimes v+p\Ibb, v)$.
 The entropy $\eta(U) = \frac 12 \rho|v|^2$ and the entropy flux  is $q(U)=\left(\frac{\rho|v|^2}{2}+p\right)v$.
Thus, assuming the usual slip boundary condition ${v}\cdot n=0$ on $\partial\Omega$, we have $q(U)\cdot n=0$ on $\partial\Omega$.

The function $B$ has the form $B(U)=(-\frac{1}{2}|v|^2, v, p)$. Again, as in the case of incompressible Euler system, $B$ does not have a linear growth. However we cannot repeat the same reasoning as for the incompressible Euler system as $G$ is not linear in the first row. 
One could add  assumptions on boundedness of appropriate quantities, however
we will proceed differently. The system will be rewritten in different variables to provide that the row of $G$ corresponding to the first component of $B$ will be linear. Thus we choose  $U=(\rho, m, p)$, where $m=\rho v$. If $\rho\ge\underline\rho>0$, then the system can be rewritten in the new variables as follows
  \begin{align}
     \label{eq:inho_eul_div_2}
    \begin{aligned}
        \partial_t \rho+ \diverg_{x}m &= 0,\\
        \partial_t m + \diverg_x\left(\frac{m\otimes m}{\rho}+p\Ibb\right)  &= 0,\\
        \diverg_x v&=0.
    \end{aligned}
\end{align}
Here $A(U)=(\rho,m,0)$ and $F(U)=(m, \frac{m\otimes m}{\rho}+p\Ibb,v)$. Moreover, $\eta(U)= \frac{|m|^2}{2\rho}$ and an entropy flux $q(U)=\left(
            \frac{|m|^2}{2\rho}
            +p
        \right)
    \frac{m}{\rho}$. 
Then the function $B$ in these variables has the form $B(U)= \left( -\frac{|m|^2}{2\rho^2},\frac{m}{\rho},p\right)$ and even though $B$ is does not have a  linear growth in the first component, but $G$ will not produce error terms in the mollification procedure in the corresponding row. As $\rho$ is bounded away from zero, it provides the linear growth of the second component of $B$.

A corresponding Onsager-type statement for inhomogeneous incompressible Euler on the torus, i.e. $\Omega=\Tbb^d$ was stated in~\cite{FGSW}, see also an analogous result for inhomogeneous incompressible Navier-Stokes equations~\cite{LSh}. As the result for  inhomogeneous incompressible Euler in~\cite{FGSW} is stated in a way that allows to trade the Besov regularity between the velocity field and density/momentum, we recall it here: 
\begin{theorem}[Th. 3.1 from \cite{FGSW}] \label{inhomonsager}
Let $(\rho, v, p)$ be a solution of~\eqref{eq:NE1} in the sense of distributions. Assume
\begin{equation}\label{besovhypo}
v\in B_p^{\alpha,\infty}(\Omega\times(0,T)),\hspace{0.3cm}\rho, \rho v\in B_q^{\beta,\infty}(\Omega\times(0,T)),\hspace{0.3cm}p\in L^{p^*}_{loc}(\Omega\times(0,T))
\end{equation}
for some $1\leq p,q\leq\infty$ and $0\leq\alpha,\beta\leq1$ such that
\begin{equation}\label{exponenthypo}
\frac{2}{p}+\frac{1}{q}=1,\hspace{0.3cm}\frac{1}{p}+\frac{1}{p^*}=1,\hspace{0.3cm}2\alpha+\beta>1.
\end{equation}
Then the energy is locally conserved, i.e.
\begin{equation}\label{localenergy}
\partial_t\left(\frac{1}{2}\rho|v|^2\right)+\diverg\left[\left(\frac{1}{2}\rho|v|^2+p\right)v\right]=0
\end{equation}
in the sense of distributions on $\Omega\times(0,T)$.
\end{theorem}
 Note that although this result is stated in the variables $(\rho, v,p)$, but there is an additional requirement that momentum $\rho v$ is an element of Besov space. 

Observe that this theorem can be extended to cases where the problem is considered in any open set $\Omega \subset \Rbb^d$ . Moreover under the condition
$$
\lim_{\varepsilon \rightarrow 0} \int_0^T\frac{1}{\epsilon}\int_{\frac{\epsilon}{4} \le d(x,\del\Omega) \le \frac{\epsilon}{2}}\left|\left(\rho \frac {|v|^2}2 + p\right) v(t,x)\cdot n(\sigma( x))\right|\, dxdt =0
$$ one has the global energy conservation for $0<t<T$.


 To conclude, taking into account the above discussion and Remark~\ref{r42}, we formulate the theorem, which follows from the general result presented in Section~\ref{localtoglobal}. 
 \begin{theorem}
 Let $(\rho,m,p)\in \underline{B}_{3,\textit{VMO}}^{1/3}(\Omega\times [0,T])\times \underline{B}_{3,\textit{VMO}}^{1/3}(\Omega\times [0,T])\times  L^3_{loc}(\Omega\times [0,T])$ be a solution to~\eqref{eq:inho_eul_div_2}.
 Moroever, let 
 $$
\lim_{\varepsilon \rightarrow 0} \int_0^T\frac{1}{\epsilon}\int_{\frac{\epsilon}{4} \le d(x,\del\Omega) \le \frac{\epsilon}{2}}\left|\left(
            \frac{|m|^2}{2\rho}
            +p
        \right)
    \frac{m}{\rho} \cdot  n(\sigma( x))\right|\, dxdt =0.
 $$
 Then the energy is globally conserved, i.e., 
\begin{equation}
\frac{d}{dt}\int_\Omega  \frac{|m|^2}{2\rho}dx=0
\end{equation}
in the sense of distributions.
 \end{theorem}
  
\subsection{Compressible Euler system}\label{con}
We consider the compressible Euler equations in the following form
\begin{align}
 \label{eq:comp_eul}
    \begin{aligned}
        \partial_t \rho+ \diverg_{x}(\rho v) &= 0, \\
        \partial_t (\rho v) + \diverg_x(\rho v\otimes v +p(\rho)\Ibb) &= 0,
    \end{aligned}
\end{align}
for an unknown vector field $v\colon \Omega\times[0,T] \to \R^n$
and scalar  $\rho\colon  \Omega\times[0,T]  \to \R$. The function $p\colon [0,\infty)\to\R$ is given.  Let $P$ be the so-called pressure potential given by
\begin{equation}
P(\rho)=\rho\int_1^\rho\frac{p(z)}{z^2} dz.
\end{equation}
Let $(\rho,v)\in \underline{B}_{3,\textit{VMO}}^{1/3}(\Omega\times [0,T])\times \underline{B}_{3,\textit{VMO}}^{1/3}(\Omega\times [0,T])$
be a weak solution to \eqref{eq:comp_eul}.
To get the conservation of the energy, we multiply \eqref{eq:comp_eul} with
\begin{equation*}
    { B}(\rho,v)=\left( P'(\rho)-\frac{1}{2}|v|^2,  v\right)
\end{equation*}
and obtain
\begin{equation}
 \label{eq:comp_cons}
    \partial_t\left(
        \frac{1}{2}\rho |v|^2 +  P(\rho)
              \right)
    +\diverg_x\left[
        \left(
            \frac{1}{2}\rho |v|^2
            +P(\rho)+p(\rho)
        \right)
    v\right]
    =0.
\end{equation}
In the variables $U=(\rho,v)$, in correspondence to the notation \eqref{FA}, we have
\begin{equation}
A(U)=(\rho, \rho v), \quad F(U)=(\rho v, \ \rho v\otimes v + p(\rho) \Ibb)
\end{equation}
and
\begin{equation}
\eta(U)=\frac{1}{2}|v|^2+P(\rho), \quad q(U)= \left(
            \frac{1}{2}\rho |v|^2
            +P(\rho)+p(\rho)
        \right)
    v.
\end{equation}
The entropy flux function $q(U)$ is in the form of a product of a scalar function and $v$, say $q(\rho, v)=\tilde q(\rho,v) v$.
Thus the condition $q(U)\cdot n=0$ on the boundary is equivalent to $v\cdot n=0$ on the boundary.

 If $\rho\ge\underline\rho>0$ the compressible Euler system can be rewritten with respect to the quantities $\rho$ and momentum $m = \rho v$ as follows
    \begin{align}
     \label{eq:comp_eul_div_2}
    \begin{aligned}
        \partial_t \rho+ \diverg_{x}m &= 0,\\
        \partial_t m + \diverg_x\left(\frac{m\otimes m}{\rho}+p(\rho)\Ibb\right)  &= 0,
    \end{aligned}
\end{align}
A suitable choice of  ${ B}$ is then \begin{equation}
    { B}(\rho,m) = \left( P'(\rho)+\frac{|m|^2}{2\rho^2},\frac{m}{\rho}\right),
\end{equation}
which leads to the companion law
\begin{equation}
    \label{eq:comp_cons_2}
    \partial_t\left(
        \frac{|m|^2}{2\rho}  + P(\rho)
              \right)
    +\diverg_x\left[
        \left(
            \frac{|m|^2}{2\rho}
            + P(\rho)+p(\rho)
        \right)
    \frac{m}{\rho}\right]
    =0.
\end{equation}
The flux function $q(U)$ is in the form of a product of a scalar function and $m/\rho$, $q(\rho, m)=\tilde q(\rho,m) \frac{m}{\rho}.$
Thus the condition $q(U)\cdot n=0$ on the boundary is equivalent to $m\cdot n=0$ on the boundary.

%

 To conclude, taking into account the above discussion and Remark~\ref{r42}, we formulate the result in a bounded domain $\Omega$, which follows from the general result presented in Section~\ref{localtoglobal}. 
 \begin{theorem}
 Let $(\rho,m)\in L^3(0,T;\underline{B}_{3,\textit{VMO}}^{1/3}(\Omega))\times L^3(0,T;\underline{B}_{3,\textit{VMO}}^{1/3}(\Omega))$ be a solution to~\eqref{eq:comp_eul_div_2}.
 Moroever, let 
 $$
\lim_{\varepsilon \rightarrow 0} \int_0^T\frac{1}{\epsilon}\int_{\frac{\epsilon}{4} \le d(x,\del\Omega) \le \frac{\epsilon}{2}}\left|\left(
            \frac{|m|^2}{2\rho}
            + P(\rho)+p(\rho)
        \right)
    \frac{m}{\rho} \cdot  n(\sigma( x))\right|\, dxdt =0
 $$
 Then the energy is globally conserved, i.e., 
\begin{equation}
\frac{d}{dt}\int_\Omega \left(
        \frac{|m|^2}{2\rho}  + P(\rho)
              \right)dx=0
\end{equation}
in the sense of distributions.
 \end{theorem}

We recall in detail the result from~\cite{FGSW}, as again, similar as in the case of the inhomogeneous incompressible Euler system, the particular form of the function $A$ allows for the interplay between the Besov regularity of particular terms, i.e., the exponents $\alpha$ and $\beta$. This result, too, was only stated for the system on the torus, i.e. $\Omega=\Tbb^d$. 

\begin{theorem}[Th. 4.1 from \cite{FGSW}]\label{compressibleonsager}
Let $\rho$, $v$ be a solution of~\eqref{eq:comp_eul} in the sense of distributions. Assume
\begin{equation*}
v\in B^{\alpha}_{3,\infty}(\Omega\times(0,T)),\hspace{0.3cm}\rho, \rho v\in B^{\beta}_{3,\infty}(\Omega\times(0,T)),\hspace{0.3cm}
0 \leq \underline{\rho} \leq \rho \leq \overline{\rho} \ \mbox{a.a. in } \Omega\times(0,T),
\end{equation*}
for some constants $\underline{\vr}$, $\overline{\vr}$, and
$0\leq\alpha,\beta\leq1$ such that
\begin{equation}\label{alphabeta}
\beta > \max \left\{ 1 - 2 \alpha; \frac{1 - \alpha}{2} \right\}.
\end{equation}
Assume further that $p \in C^2[\underline{\vr}, \overline{\vr}]$, and, in addition
\begin{equation}\label{pressure}
p'(0) = 0 \ \mbox{as soon as}\ \underline{\vr} = 0.
\end{equation}
Then the energy is locally conserved, i.e.
\begin{equation*}
\partial_t\left(\frac{1}{2}\rho|v|^2+P(\rho)\right)+\diverg\left[\left(\frac{1}{2}\rho|v|^2+p(\rho)+P(\rho)\right)v\right]=0
\end{equation*}
in the sense of distributions on $\Omega\times(0,T)$.
\end{theorem}

In $\alpha<\frac{1}{3}$, then  in case of incompressible Euler system we know there would exist $C^\alpha$ solutions that do not conserve energy. 
The above theorem indicates that  in case of compressible model, if together with the information $\alpha<\frac{1}{3}$ we know that the density $\rho$ is fine enough to provide that the product $\rho v$ is in a better space  (i.e. with an exponent $\beta$ sufficiently high), then the energy is conserved even for such low regular velocity fields.  

Similarly to the incompressible case elaborated in~\cite{BTW18}, a possible application concerns the Navier-Stokes-to-Euler limit, as viscosity tends to zero: A vanishing viscosity sequence of solutions to the compressible Navier-Stokes equations that \emph{uniformly} satisfies the interior Besov condition and the normal condition~\eqref{boundaryass} near the boundary will converge to a solution of the compressible Euler equations. This is particularly important as these requirements are consistent with the possible formation of a boundary layer, which affects only the regularity of the \emph{tangential} velocity component.

\subsection{Polyconvex elasticity}\label{polycon}

In this section we first consider a quasi-linear wave equation that may be interpreted as a model of nonlinear elastodynamics, when we understand $y\colon \Omega\times{\R}^+ \to{\R}^3$ as a displacement vector
\begin{equation}
\label{mainI}
\frac{\partial^2 y}{\partial t^2}=\diverg_x S(\nabla y).
\end{equation}
In the above equation $S$ is a gradient of some function ${\mathcal G}:{\mathbb M}^{3\times 3} \to [0,\infty)$.
We rewrite the equation as a system, introducing the notation $v_i=\partial_ty_i$ and $\Fbb_{i j}=\frac{\partial y_i}{\partial x_j}$. Then
$U=(v,\Fbb)$ solves the system
\begin{equation}\label{poly}
\begin{aligned}
\frac{\partial v_i}{\partial t}&=\frac{\partial}{\partial x_j}\left(\frac{\partial {\mathcal G}}{\partial \Fbb_{i j}}\right),\\
\frac{\partial \Fbb_{i j}}{\partial t}&=\frac{\partial v_i}{\partial x_j}.
\end{aligned}
\end{equation}
With $A(U)\equiv id$ and $F(U)=\left(\frac{\partial {\mathcal G}}{\partial \Fbb_{i j}}, v\right)$ we have an entropy
$\eta(U)=\frac{1}{2}|v|^2+{\mathcal G}(\Fbb)$ and an entropy flux $q_j(U)=v_i\frac{\partial {\mathcal G}(\Fbb)}{\partial \Fbb_{i j}}$.
 Then the suitable choice of function $B$ is $B=(v, \frac{\partial {\mathcal G}(\Fbb)}{\partial \Fbb_{i j}})$. A typical conditions that are assumed on ${\mathcal G}$ are the following (see e.g. Section 2.2 in \cite{DST}):
\begin{equation}
{\mathcal G}\in C^3, \ |D^3{\mathcal G}(F)|\le M \mbox{ for some } M>0,
\end{equation}
\begin{equation}
{\mathcal G}(F)=g_0(F)+\frac{1}{2}|F|^2\ \mbox{where }\lim_{|F|\to\infty}\frac{g_0(F)}{1+|F|^2}=0
\end{equation}
and 
\begin{equation}
\lim_{|F|\to\infty}\frac{\frac{\partial {\mathcal G(F)}}{\partial\Fbb_{i j}}}{1+|F|^2}=0.
\end{equation}

Then again $B$ may not satisfy the requirement of linear growth, however the problem does not arise here, as the corresponding  row of the flux is linear. Under these assumptions entropy and entropy flux satisfy condition~\eqref{locest}.

One of the  natural boundary conditions $S n=0$ on $\partial \Omega$  (i.e.$\frac{\partial {\mathcal G}}{\partial \Fbb_{i j}}$ vanishes in the normal direction), so-called zero traction boundary condition,  implies that $q(U)\cdot n=0$ on~$\partial \Omega$. Another boundary conditions that is often considered is the Dirichlet boundary condition $v=0$ on $\partial \Omega$. In that case again $q(U)\cdot n=0$ on~$\partial \Omega$. 

Taking into account Remark~\ref{extension} and the above discussion we are ready to state the result in a bounded domain $\Omega$, which can be proved using the general result of Section~\ref{localtoglobal}. 
\begin{theorem}
Let $(v, \Fbb) \in
L^3(0,T;\underline{B}_{3,\textit{VMO}}^{1/3}(\Omega))\times L^3(0,T;\underline{B}_{3,\textit{VMO}}^{1/3}(\Omega))$ be a solution to \eqref{poly}. Moroever, let 
 $$
\lim_{\varepsilon \rightarrow 0} \int_0^T\frac{1}{\epsilon}\int_{\frac{\epsilon}{4} \le d(x,\del\Omega) \le \frac{\epsilon}{2}}\left|
            v_i\frac{\partial {\mathcal G}(\Fbb)}{\partial \Fbb_{i j}} \cdot  n(\sigma( x))\right|\, dxdt =0.
 $$
Then the energy is globally conserved, i.e., 
\begin{equation}
\frac{d}{dt}\int_\Omega \left(\frac{1}{2}|v|^2+{\mathcal G}(\Fbb) \right)dx=0
\end{equation}
in the sense of distributions.

\end{theorem}
\begin{remark}
For system~\eqref{poly}, in the spirit of earlier examples, one could formulate conditions allowing to distinguish among different  regularity requirements for $v$ and $\Fbb$. Since the nonlinearity only appears in $\Fbb$, then 
to provide the conservation of entropy we could assume that $(v, \Fbb) \in
L^3(0,T;\underline{B}_{3,\textit{VMO}}^{\alpha}(\Omega))\times L^3(0,T;\underline{B}_{3,\textit{VMO}}^{\beta}(\Omega))$ with $\beta=\frac{1-\alpha}{2}$. 
\end{remark}


In the case of elastodynamics  we regard  $S$ as the Piola-Kirchoff stress tensor
obtained as the gradient of a stored energy function,
$S = \frac{\partial W}{\partial {\mathbb F}}$.  A natural assumption is that
$W$ is polyconvex, that is  $W({\mathbb F}) = {\mathcal G} ( \Phi({\mathbb F}))$ where
${\mathcal G}:{\mathbb M}^{3\times 3}\times{\mathbb M}^{3\times 3}\times \R \to [0,\infty)$
is a strictly convex function and $\Phi({\mathbb F}) = ({\mathbb F} ,\cof {\mathbb F}, \det {\mathbb F})\in
{\mathbb M}^{3\times 3}\times{\mathbb M}^{3\times 3}\times \R$
stands for the vector of null-Lagrangians: ${\mathbb F}$, the cofactor matrix $\cof {\mathbb F}$
and the determinant $\det {\mathbb F}$.

The system can be embedded into the following symmetrizable hyperbolic system in a new dependent variable $\Xi=({\mathbb F},Z,w)$,  see e.g.~\cite{Dafermos1985} and~\cite{DST}, taking values in
${\mathbb M}^{3\times 3}\times {\mathbb M}^{3\times 3}\times\R$:
\begin{equation}
\begin{aligned}
\frac{\partial v_i}{\partial t}&=\frac{\partial}{\partial x_j}\left(\frac{\partial {\mathcal G}}{\partial\Xi^A}(\Xi)\frac{\partial\Phi^A}{\partial
{\mathbb F}_{i j}}({\mathbb F})\right),\\
\frac{\partial\Xi^A}{\partial t}&=\frac{\partial}{\partial
x_j}\left(\frac{\partial\Phi^A}{\partial
{\mathbb F}_{i j}}({\mathbb F})v_i\right),
\end{aligned}
\end{equation}
and hence  for $U=(v,{\mathbb F},Z,w)$ we have
\begin{equation}
A(U)=id, \quad F(U)=\left(\frac{\partial {\mathcal G}}{\partial\Xi^A}(\Xi)\frac{\partial\Phi^A}{\partial
{\mathbb F}_{i j}}({\mathbb F}), \frac{\partial\Phi^A}{\partial
{\mathbb F}_{i j}}({\mathbb F})v_i\right).
\end{equation}
This system admits the following entropy-entropy flux pair
\begin{equation}
\begin{aligned}
\eta(v,{\mathbb F},Z,w)&=\frac{1}{2}|v|^2+{\mathcal G}({\mathbb F},Z,w),\\
q_j(v,{\mathbb F},Z,w)&=v_i\,\frac{\partial {\mathcal G}}{\partial\Xi^A}(\Xi)\frac{\partial\Phi^A}{\partial {\mathbb F}_{i j}}({\mathbb F}).
\end{aligned}
\end{equation}
Of course in this situation the conclusions follow from the ones stated in the general case described at first.

\subsection{Incompressible magnetohydrodynamics}
Let us consider the system
\begin{align}
    \label{eq:MHD}
  \left.
  \begin{aligned}
           \partial_t v + \Div_x(v\otimes  v - h\otimes h) + \nabla_x (p+\frac{1}{2}|h|^2) &= 0,
            \\
            \partial_t h + \Div_x(v\otimes h-h \otimes v)  &= 0,\\
              \Div_xv &= 0,
            \\
            \Div_xh&= 0,
           \end{aligned}
           \right.
\end{align}
for unknown vector functions $v\colon \Omega\times[0,T] \to \R^n$ and $h\colon \Omega\times[0,T] \to \R^n$ and an unknown scalar function $p\colon \Omega\times[0,T]\to \R$. It is sufficient to require that $\Div_x h$ is equal to zero at the initial time, as this information is then transported and thus we may reduce the system to $2n+1$ equations.  The system  describes the motion of an ideal electrically conducting fluid, see e.g.~\cite[Chapter VIII]{landau}.

Here $U=(v,h,p)$, $A(U)=(v,h,0)$,
and
$$F(v,h) = \left(v\otimes v-h\otimes h+(p+\frac{1}{2}|h|^2)\Ibb,v\otimes h-h \otimes v,v\right).$$
The entropy is given as $\eta = \frac12 (|v|^2 + |h|^2)$ 
and the entropy fluxes are $q= \frac 12(|v|^2 + |h|^2)v - (v\cdot h) h$. The possible choice of 
the function $B$ is the following $B=(v,h, p-\frac 12|v|^2)$. 

Assuming $v\cdot n=0$ and $h\cdot n=0$ on the boundary (see e.g. \cite{FNS2014}) provides that $q(U)\cdot n=0$ on the boundary.

Using again Remark~\ref{extension} we can state the result in a bounded domain $\Omega$.
 \begin{theorem}
 Let $(v,h,p)\in L^3(0,T;\underline{B}_{3,\textit{VMO}}^{1/3}(\Omega))\times L^3(0,T;\underline{B}_{3,\textit{VMO}}^{1/3}(\Omega))\times L^3((0,T)\times\Omega)$ be a solution to~\eqref{eq:MHD}.
 Moroever, let 
 $$
\lim_{\varepsilon \rightarrow 0} \int_0^T\frac{1}{\epsilon}\int_{\frac{\epsilon}{4} \le d(x,\del\Omega) \le \frac{\epsilon}{2}}\left|\left(
            \frac 12(|v|^2 + |h|^2)v - (v\cdot h) h
        \right)
    \cdot  n(\sigma( x))\right|\, dxdt =0.
 $$
 Then the energy is globally conserved, i.e., 
\begin{equation}
\frac{d}{dt}\int_\Omega \frac12 (|v|^2 + |h|^2)dx=0
\end{equation}
in the sense of distributions.
 \end{theorem}

 In fact, the integrability requirement on $p$ can be relaxed owing to the elliptic arguments of~\cite{BTW18}, cf.\ the remark at the end of subsection~\ref{incompEul}.

An extension of the error estimates from~\cite{CET} was proposed by Caflisch et al.\ \cite{Cafetal} to handle
the global energy conservation for incompressible magnetohydrodynamics in the case without a boundary, i.e., $\Omega=\T^d$. 
We recall this result below:

\begin{theorem}[Th. 4.1 from~ \cite{Cafetal}] \label{caflisch}
Let $d = 2,3$ and let $(v,h)$ be a weak solution of ~\eqref{eq:MHD}. Suppose that
\[v\in C([0,T],B_{3,\infty}^{\alpha}(\Omega)),\;\; h\in C([0,T],B_{3,\infty}^{\beta}(\Omega))
\]
with
\[\alpha > \frac{1}{3},\;\;\alpha+2\beta > 1.
\]
Then the following energy identity holds for any $t\in[0,T]$:
\begin{equation}\label{energyMHD}
\int_{\Omega}|v(x,t)|^2 + |h(x,t)|^2\ \dx = \int_{\Omega}|v(x,0)|^2 + |h(x,0)|^2\ \dx.
\end{equation}
\end{theorem}
The same system was studied by Kang and Lee \cite{KangLee}, who
formulated the result in the spirit of the framework of
 Cheskidov et al.~\cite{CCFS08}, with $\Omega=\R^3$:
\begin{theorem}[Th. 6 from~\cite{KangLee}]
let $(v,h)$ be a weak solution of ~\eqref{eq:MHD}. Suppose that
\[v\in L^3([0,T],B^{\alpha}_{3,c_0}(\Omega)),\;\; h\in L^3([0,T],B^{\beta}_{3,c_0}(\Omega))
\]
with
\[\alpha \geq \frac{1}{3},\;\;\alpha+2\beta \geq 1.
\]
Then ~\eqref{energyMHD} holds.
\end{theorem}

%

\subsection{Compressible magnetohydrodynamics}
We consider the system
\begin{align}
    \label{eq:MHD_diver}
    \begin{aligned}
    \rho_t+\diverg_{x}(\rho v) &= 0,\\
 \partial_t(\rho v) + \diverg_x\left(\rho v\otimes v + p(\rho)\Ibb +\frac{1}{2}|h|^2\Ibb-h\otimes h\right)&= 0,\\
    \partial_t h + \diverg_x(h \otimes v-v\otimes h)  &= 0,\\
        \diverg_{x}h &= 0,
        \end{aligned}
\end{align}
  where $v \colon\Omega\times[0,T]\to\R^n$ is the velocity field, $\rho \colon\Omega\times[0,T]\to\R$ the density of the fluid and $h\colon\Omega\times[0,T]\to\R^3$ is the magnetic field.
With ${ B}(\rho,v,h)=(P'(\rho)-1\slash 2 |v|^2,v,h,-h\cdot v)$, the conservation of the total energy reads:
\begin{align}
    \label{eq:MHD_en}
    &\partial_t\left(
            \frac{1}{2}\rho|v|^2 +  P(\rho)+ \frac{1}{2}|h|^2
    \right)+\diverg_x\left[
    \left(\frac{1}{2} \rho|v|^2+P(\rho)+p(\rho)+|h|^2\right) v -(v\cdot h)h
           \right]=0.
\end{align}
Assuming again $v\cdot n=0$ and $h\cdot n=0$  on the boundary provides that $q(U)\cdot n=0$ on the boundary.

In the case of compressible magnetohydrodynamics, we can act in a similar fashion as in the case of the compressible Euler system and formulate the equations in different variables $\rho,m,h$, where again $m$ is the momentum, i.e., $m=\rho v$:
\begin{align}
   \label{eq:MHD_diver-m}
   \begin{aligned}
    \rho_t+\diverg_{x}m &= 0,\\
 \partial_t m + \diverg_x\left(\frac{m\otimes m}{\rho} + p(\rho)\Ibb +\frac{1}{2}|h|^2\Ibb-h\otimes h\right)&= 0,\\
    \partial_t h + \diverg_x\left(\frac{h \otimes m}{\rho}-\frac{m\otimes h}{\rho}\right)  &= 0,\\
        \diverg_{x}h &= 0.
        \end{aligned}
\end{align}


Similarly,  if $\rho>\underline\rho>0$, we can state the result in a bounded domain $\Omega$.
 \begin{theorem}
 Let $(\rho,m,h)\in L^3(0,T;\underline{B}_{3,\textit{VMO}}^{1/3}(\Omega))\times L^3(0,T;\underline{B}_{3,\textit{VMO}}^{1/3}(\Omega))\times L^3(0,T;\underline{B}_{3,\textit{VMO}}^{1/3}(\Omega))$ be a solution to~\eqref{eq:MHD_diver-m}.
 Moroever, let 
 $$
\lim_{\varepsilon \rightarrow 0} \int_0^T\frac{1}{\epsilon}\int_{\frac{\epsilon}{4} \le d(x,\del\Omega) \le \frac{\epsilon}{2}}\left|\left(\left(
            \frac{|m|^2}{2\rho}
            + P(\rho)+p(\rho)+|h|^2
        \right)
    \frac{m}{\rho} -\left(\frac{m}{\rho}\cdot h\right) h
        \right)
    \cdot  n(\sigma( x))\right|\, dxdt =0.
 $$
 Then the energy is globally conserved, i.e., 
\begin{equation}
\frac{d}{dt}\int_\Omega  \left(
        \frac{|m|^2}{2\rho}  + P(\rho)+\frac 12|h|^2
              \right)dx=0
\end{equation}
in the sense of distributions.
 \end{theorem}



 \section*{Acknowledgements}
The authors would like to thank the   Wolfgang Pauli Institute, Vienna, for the warm and kind hospitality  where the authors completed this work. The work of E.S.T.\ was supported in part by the Einstein Stiftung/Foundation - Berlin, through the Einstein Visiting Fellow Program, and by the John Simon Guggenheim Memorial Foundation. P.G. and A.\'S.-G. received support from the National Science Centre (Poland), 2015/18/MST1/00075.

\end{document}